\documentclass[12pt,twoside,a4paper,reqno]{amsart}
\usepackage{bullet}
\usepackage{graphics}
\usepackage{amssymb}
\usepackage{eucal}
\usepackage{amsmath}
\info{*}{**}{*}{****}
\title{ON THE PROPERTIES OF THE ARON-BERNER REGULARITY OF BOUNDED TRI-LINEAR MAPS}

\author{Neda Akhlaghi}
\address{$^1$Independent Researcher, e-mail: {\tt Neda.akhlaghi1365@gmail.com}}

\author{Kazem Haghnejad Azar}
\address{$^2$Department  of  Mathematics  and  Applications, University of Mohaghegh Ardabili, Ardabil, Iran, e-mail: {\tt Haghnejad@aut.ac.ir}}

\author{Abotaleb Sheikhali}
\address{$^3$Independent Researcher, e-mail: {\tt Abotaleb.sheikhali.20@gmail.com}}

\begin{document}
\pagestyle{headings}
\maketitle

\begin{abstract}

\mm \mm

{\it \quad\qu  Let $f:X\times Y\times Z\longrightarrow W $ be a bounded tri-linear map on normed spaces. We say that $f$ is close-to-regular when $f^{t****s}=f^{s****t}$ and $f$ is Aron-Berener regular when all natural extensions are equal. In this manuscript, we have some results on the Aron-Berner regular maps. We investigate the relation between Arens regularity of bounded bilinear maps and Aron-Berner regularity of bounded tri-linear maps. We  also give a simple criterion for the Aron-Berner regularity of tri-linear maps.}
\end{abstract}

\begin{Keywords}
Arens product, Aron-Berner regular, Close-to-regular, Tri-linear map.
\end{Keywords}

2010 Mathematics Subject Clasification: 46H25; 46H20; 17C65

\section{Introduction}
\label{intro}
In \cite{Arens2},  Richard Arens showed  that a bounded bilinear map $m:X\times Y\longrightarrow Z $ on normed spaces, has two natural  different extensions $m^{***}$, $m^{r***r}$ from $X^{**}\times Y^{**}$ into $Z^{**}$. When these extensions are equal, $m$ is called Arens regular. A Banach algebra $A$ is said to be Arens regular, if its product $\pi(a,b)=ab$ considered as a bilinear mapping $\pi : A\times A\longrightarrow A$ is Arens regular. 
For a discussion of Arens regularity for  bounded bilinear maps and Banach algebras, see \cite{Arikan1}, \cite{Eshaghi}, \cite{Haghnejad Azar}, \cite{Mohamadzadeh}, \cite{ulger} and \cite{sheikh}. For example, every $C^{*}$-algebra is Arens regular, see \cite{Civin and Yood}. Arens himself proved in \cite{Arens2} that the convolution semigroup algebra $l^{1}$  is not Arens regular. Also Yong proved in \cite{young} that $L^{1}(G)$ is Arens regular if and only if G is finite. 

Let $X, Y, Z$ and $W$ be normed spaces and $f:X\times Y\times Z\longrightarrow W $ be a bounded tri-linear mapping. One of the natural extensions  of $f$ can be derived by the following procedure:
\begin{enumerate}
\item $f^{*}:W^{*}\times X\times Y\longrightarrow Z^{*}$, given by $\langle f^{*}(w^{*},x,y),z\rangle=\langle w^{*},f(x,y,z)\rangle$ where $x\in X, y\in Y, z\in Z, w^{*}\in W^{*}$. 

The map $f^*$ is a bounded tri-linear mapping and is said  the adjoint of $f$.

\item $f^{**}=(f^*)^*:Z^{**}\times W^{*}\times X\longrightarrow Y^{*}$, given by $\langle f^{**}(z^{**},w^{*},x),y\rangle=\langle z^{**}, f^{*}(w^{*},x,y)$ where $x\in X, y\in Y, z^{**}\in Z^{**}, w^{*}\in W^{*}$.

\item $f^{***}=(f^{**})^*:Y^{**}\times Z^{**}\times W^{*}\longrightarrow X^{*}$, given by $\langle f^{***}(y^{**},z^{**},w^{*}),x\rangle=\langle y^{**},  f^{**}(z^{**},w^{*},x) \rangle$ where $x\in X, y^{**}\in Y^{**}, z^{**}\in Z^{**}, w^{*}\in W^{*}$.

\item $f^{****}=(f^{***})^*:X^{**}\times Y^{**}\times Z^{**}\longrightarrow W^{**}$, given by $\langle f^{****}(x^{**},
y^{**},z^{**})$, $w^{*}\rangle =\langle x^{**}, f^{***}(y^{**},z^{**},w^{*}) \rangle$ where $x^{**}\in X^{**}, y^{**}\in Y^{**}, z^{**}\in Z^{**}, w^{*}\in W^{*}$.
\end{enumerate}
The bounded tri-linear map $f^{****}$ is the extension of $f$ such that the maps 
\begin{eqnarray*}
&&x^{**}\longrightarrow f^{****}(x^{**},y^{**},z^{**}):X^{**}\longrightarrow W^{**},\\
&&y^{**}\longrightarrow f^{****}(x,y^{**},z^{**}):Y^{**}\longrightarrow W^{**},\\
&&z^{**}\longrightarrow f^{****}(x,y,z^{**}):Z^{**}\longrightarrow W^{**},
\end{eqnarray*}
are weak$^{*}-$weak$^{*}$ continuous for each $x\in X, y\in Y, x^{**}\in X^{**}, y^{**}\in Y^{**}$ and $z^{**}\in Z^{**}$. See \cite{sheikhd}.
Now let 
\begin{eqnarray*}
&&f^i:Y\times X\times Z\longrightarrow W : f^i(y,x,z)=f(x,y,z),\\
&&f^j:X\times Z\times Y\longrightarrow W : f^j(x,z,y)=f(x,y,z),\\
&&f^r:Z\times Y\times X\longrightarrow W : f^r(z,y,x)=f(x,y,z),\\
&&f^t:Z\times X\times Y\longrightarrow W : f^t(z,x,y)=f(x,y,z),\\
&&f^s:Y\times Z\times X\longrightarrow W : f^s(y,z,x)=f(x,y,z),
\end{eqnarray*}
be the flip maps of $f$, for every $x\in X ,y\in Y$ and $z\in Z$. The flip maps of $f$ are bounded tri-linear maps. 
For natural extensions of $f$ we have
\begin{eqnarray*}
&& f^{****}(x^{**},y^{**},z^{**})=w^{*}-\lim\limits_\alpha \lim\limits_\beta \lim\limits_\gamma f(x_\alpha,y_\beta,z_{\gamma}),\\
&& f^{i****i}(x^{**},y^{**},z^{**})=w^{*}-\lim\limits_\beta \lim\limits_\alpha \lim\limits_\gamma f(x_\alpha,y_\beta,z_{\gamma}),\\
&& f^{j****j}(x^{**},y^{**},z^{**})=w^{*}-\lim\limits_\alpha \lim\limits_\gamma \lim\limits_\beta f(x_\alpha,y_\beta,z_{\gamma}),\\
&& f^{r****r}(x^{**},y^{**},z^{**})=w^{*}-\lim\limits_\gamma \lim\limits_\beta \lim\limits_\alpha f(x_\alpha,y_\beta,z_{\gamma}),\\
&& f^{t****s}(x^{**},y^{**},z^{**})=w^{*}-\lim\limits_\gamma \lim\limits_\alpha \lim\limits_\beta f(x_\alpha,y_\beta,z_{\gamma}),\\
&& f^{s****t}(x^{**},y^{**},z^{**})=w^{*}-\lim\limits_\beta \lim\limits_\gamma \lim\limits_\alpha f(x_\alpha,y_\beta,z_{\gamma}),
\end{eqnarray*}
where $(x_{\alpha}), (y_{\beta})$ and $(z_{\gamma})$ are nets in $X, Y$ and $Z$  which converge to $x^{**}\in X^{**},y^{**}\in Y^{**}$ and $z^{**}\in Z^{**}$  in the $w^{*}-$topologies, respectively.
 
A bounded tri-linear map $f$ is said to be close-to-regular if $f^{t****s}=f^{s****t}$. Some characterizations for the close-to-regularity of bounded tri-linear map $f$  can be found in \cite{sheikhoo}, authors showed that  if $Y$ is reflexive, then $f:X\times Y\times Z\longrightarrow W$ is close-to-regular. Also the bounded tri-linear map $f$ is said to be Aron–Berner regular when all natural extensions are equal, that is, $f^{i****i}=f^{j****j}=f^{r****r}=f^{****}=f^{t****s}=f^{s****t}$ holds. For example in \cite{Khosravi} it has been shown that the tri-linear map $f:X\times X\times X\longrightarrow X$ deﬁned by
$$f(x_{1}, x_{2}, x_{3})=\langle \phi,x_{1}\rangle \langle \psi ,x_{2}\rangle x_{3}  \ \ \ \ \ \ \      (x_{1}, x_{2}, x_{3} \in X),$$
for every $\phi , \psi \in X^{*}$ is Aron–Berner regular. If $f$ is Aron–Berner regular, then trivially $f$ is close-to-regular. 

\section{Aron-Berner regularity of bounded tri-linear maps}
\begin{theorem}\label{2.1}
Let $f:X\times Y\times Z\longrightarrow W$ be a bounded tri-linear map. Then
\begin{enumerate}
\item $f$ is Aron-Berner regular if and only if $f^{t****s}=f^{****}=f^{s****t}.$ 

\item $f$ is Aron-Berner regular if and only if $f^{i****i}=f^{j****j}=f^{r****r}.$
\end{enumerate}
\end{theorem}
\begin{proof}
We prove only (1), the other part has the same argument. If $f$ is Aron-Berner regular then $f^{t****s}=f^{****}=f^{s****t}.$\\
For the converse, suppose that $(x_{\alpha}), (y_{\beta})$ and $(z_{\gamma})$ are nets in $X, Y$ and $Z$ converge to $x^{**}\in X^{**},y^{**}\in Y^{**}$ and $z^{**}\in Z^{**}$  in the $w^{*}-$topologies, respectively. If $f^{****}=f^{s****t}$, then for every $w^{*}\in W^{*}$ we have
\begin{eqnarray*}
&&\langle f^{****}(x^{**},y^{**},z^{**}),w^{*}\rangle =\lim\limits_\beta\lim\limits_\alpha\lim\limits_\gamma \langle f(x_{\alpha},y_{\beta},z_{\gamma}),w^{*}\rangle\\
&&=\lim\limits_\beta\lim\limits_\alpha\lim\limits_\gamma \langle f^{*}(w^{*},x_{\alpha},y_{\beta}),z_{\gamma}\rangle=\lim\limits_\beta\lim\limits_\alpha \langle z^{**}, f^{*}(w^{*},x_{\alpha},y_{\beta})\rangle\\
&&=\lim\limits_\beta\lim\limits_\alpha\langle f^{**}(z^{**},w^{*},x_{\alpha}),y_{\beta}\rangle=\lim\limits_\beta\lim\limits_\alpha\langle f^{***}(y_{\beta},z^{**},w^{*}),x_{\alpha}\rangle\\
&&=\lim\limits_\beta\langle x^{**},f^{***}(y_{\beta},z^{**},w^{*})\rangle=\lim\limits_\beta\langle f^{****}(x^{**},y_{\beta},z^{**}),w^{*}\rangle\\
&&=\lim\limits_\beta\langle f^{s****t}(x^{**},y_{\beta},z^{**}),w^{*}\rangle=\lim\limits_\beta\langle f^{s****}(y_{\beta},z^{**},x^{**}),w^{*}\rangle\\
&&=\lim\limits_\beta\langle y_{\beta},f^{s***}(z^{**},x^{**},w^{*})\rangle=\langle y^{**},f^{s***}(z^{**},x^{**},w^{*})\rangle\\
&&=\langle f^{s****}(y^{**},z^{**},x^{**}),w^{*}\rangle =\langle f^{s****t}(x^{**},y^{**},z^{**}),w^{*}\rangle.
\end{eqnarray*}
Therefore $$f^{i****i}=f^{s****t}\ \ \ \ \ \ \ \ \ \ \ \ \ \ \ (2-1)$$
Now if $f^{****}=f^{t****s}$, then we have 

\begin{eqnarray*}
&&\langle f^{j****j}(x^{**},y^{**},z^{**}),w^{*}\rangle =\lim\limits_\alpha\lim\limits_\gamma\lim\limits_\beta \langle f(x_{\alpha},y_{\beta},z_{\gamma}),w^{*}\rangle\\
&&=\lim\limits_\alpha\lim\limits_\gamma\lim\limits_\beta \langle f^{t}(z_{\gamma},x_{\alpha},y_{\beta}),w^{*}\rangle=\lim\limits_\alpha\lim\limits_\gamma\lim\limits_\beta \langle f^{t*}(w^{*},z_{\gamma},x_{\alpha}),y_{\beta}\rangle\\
&&=\lim\limits_\alpha\lim\limits_\gamma\langle y^{**},f^{t*}(w^{*},z_{\gamma},x_{\alpha})\rangle=\lim\limits_\alpha\lim\limits_\gamma\langle f^{t**}(y^{**},w^{*},z_{\gamma}),x_{\alpha}\rangle\\
&&=\lim\limits_\alpha\lim\limits_\gamma\langle f^{t***}(x_{\alpha},y^{**},w^{*}),z_{\gamma}\rangle=\lim\limits_\alpha\langle z^{**},f^{t***}(x_{\alpha},y^{**},w^{*})\rangle\\
&&=\lim\limits_\alpha \langle f^{t****}(z^{**},x_{\alpha},y^{**}),w^{*}\rangle=\lim\limits_\alpha \langle f^{t****s}(x_{\alpha},y^{**},z^{**}),w^{*}\rangle\\
&&=\lim\limits_\alpha \langle f^{****}(x_{\alpha},y^{**},z^{**}),w^{*}\rangle=\lim\limits_\alpha \langle x_{\alpha},f^{***}(y^{**},z^{**},w^{*})\rangle\\
&&=\langle x^{**},f^{***}(y^{**},z^{**},w^{*})\rangle=\langle f^{****}(x^{**},y^{**},z^{**}),w^{*}\rangle.
\end{eqnarray*}
Therefore $$f^{j****j}=f^{****}\ \ \ \ \ \ \ \ \ \ \ \ \ \ \ (2-2)$$
On the other hand if $f^{t****s}=f^{s****t}$, then we have 

\begin{eqnarray*}
&&\langle f^{r****r}(x^{**},y^{**},z^{**}),w^{*}\rangle =\lim\limits_\gamma\lim\limits_\beta\lim\limits_\alpha \langle f(x_{\alpha},y_{\beta},z_{\gamma}),w^{*}\rangle\\
&&=\lim\limits_\gamma\lim\limits_\beta\lim\limits_\alpha \langle f^{s}(y_{\beta},z_{\gamma},x_{\alpha}),w^{*}\rangle=\lim\limits_\gamma\lim\limits_\beta\lim\limits_\alpha \langle f^{s*}(w^{*},y_{\beta},z_{\gamma}),x_{\alpha}\rangle\\
&&=\lim\limits_\gamma\lim\limits_\beta \langle x^{**},f^{s*}(w^{*},y_{\beta},z_{\gamma})\rangle=\lim\limits_\gamma\lim\limits_\beta \langle f^{s**}(x^{**},w^{*},y_{\beta}),z_{\gamma}\rangle\\
&&=\lim\limits_\gamma\lim\limits_\beta \langle f^{s***}(z_{\gamma},x^{**},w^{*}),y_{\beta}\rangle=\lim\limits_\gamma \langle y^{**},f^{s***}(z_{\gamma},x^{**},w^{*})\rangle\\
&&=\lim\limits_\gamma \langle f^{s****}(y^{**},z_{\gamma},x^{**}),w^{*}\rangle=\lim\limits_\gamma \langle f^{s****t}(x^{**},y^{**},z_{\gamma}),w^{*}\rangle\\
&&=\lim\limits_\gamma \langle f^{t****s}(x^{**},y^{**},z_{\gamma}),w^{*}\rangle=\lim\limits_\gamma \langle f^{t****}(z_{\gamma},x^{**},y^{**}),w^{*}\rangle\\
&&=\lim\limits_\gamma \langle z_{\gamma},f^{t***}(x^{**},y^{**},w^{*})\rangle=\langle z^{**},f^{t***}(x^{**},y^{**},w^{*})\rangle\\
&&=\langle f^{t****}(z^{**},x^{**},y^{**}),w^{*}\rangle=\langle f^{t****s}(x^{**},y^{**},z^{**}),w^{*}\rangle.
\end{eqnarray*}
Therefore $$f^{r****r}=f^{t****s}\ \ \ \ \ \ \ \ \ \ \ \ \ \ \ (2-3)$$
Now Using (2-1), (2-2) and (2-3), the result holds.
\end{proof}

As an immediate consequence of Theorem \ref{2.1}, we deduce the next results.
\begin{corollary}
Let $f:X\times Y\times Z\longrightarrow W$ be a bounded tri-linear map. $f^{t****s}=f^{****}=f^{s****t}$ if and only if $f^{i****i}=f^{j****j}=f^{r****r}.$

\end{corollary}
\begin{corollary}
Let $f:X\times Y\times Z\longrightarrow W$ be a bounded tri-linear map. Then
\begin{enumerate}
\item $f$ is Aron-Berner regular if and only if 
\begin{eqnarray*}
w^{*}-\lim\limits_\gamma \lim\limits_\alpha \lim\limits_\beta f(x_\alpha,y_\beta,z_{\gamma})&=&w^{*}-\lim\limits_\alpha \lim\limits_\beta \lim\limits_\gamma f(x_\alpha,y_\beta,z_{\gamma})\\
&=&w^{*}-\lim\limits_\beta \lim\limits_\gamma \lim\limits_\alpha f(x_\alpha,y_\beta,z_{\gamma}).
\end{eqnarray*}
\item $f$ is Aron-Berner regular if and only if 
\begin{eqnarray*}
w^{*}-\lim\limits_\beta \lim\limits_\alpha \lim\limits_\gamma f(x_\alpha,y_\beta,z_{\gamma})
&=&w^{*}-\lim\limits_\alpha \lim\limits_\gamma \lim\limits_\beta f(x_\alpha,y_\beta,z_{\gamma})\\
&=&w^{*}-\lim\limits_\gamma \lim\limits_\beta \lim\limits_\alpha f(x_\alpha,y_\beta,z_{\gamma}).
\end{eqnarray*}
\end{enumerate}
\end{corollary}
\begin{corollary}
For a bounded tri-linear map $f:X\times Y\times Z\longrightarrow W$  the following statements are equivalent:
\begin{enumerate}
\item $f$ is Aron-Berner regular,

\item $f^{t}$ and $f^{s}$ are close-to-regular,

\item $f^{i}$ and $f^{r}$ are close-to-regular,

\item $f^{i}$ and $f^{j}$ are close-to-regular.
\end{enumerate}
\end{corollary}
\begin{proof}
We prove only $(1)\Rightarrow (3)$ and $(3)\Rightarrow (1)$, the other parts $(1)\Rightarrow (2)$ and $(2)\Rightarrow (1)$ , $(1)\Rightarrow (4)$ and $(4)\Rightarrow (1)$ have the same argument.

Since the $f^{rt}=f^{j}=f^{sr}$ and $f^{rs}=f^{i}=f^{tr}$, thus $f^{r}$ is close-to-regular  if and only if
$$f^{rt****s}=f^{rs****t}\Leftrightarrow f^{rt****sr}=f^{rs****tr}\Leftrightarrow f^{j****j}=f^{i****i}.$$
Similarly, $f^{i}$ is close-to-regular  if and only if $f^{j****j}=f^{r****r}.$ Now by theorem \ref{2.1} completes the proof.
\end{proof}
\begin{corollary}
Let $f:X\times Y\times Z\longrightarrow W$ be a bounded tri-linear map. Then $f$ is Aron-Berner regular if and only if $f^{****s**t}(Y^{**},X^{**},W^{*})=f^{s**t****}(Y^{**},X^{**},W^{*})$ and $f^{t**s****}(W^{*},Z^{**},Y^{**})=f^{****t**s}(W^{*},Z^{**},Y^{**})$.
\end{corollary}
\begin{proof}
Let $f$ be Aron-Berner regular, then $f^{****}=f^{s****t}$. If $(y_{\beta})$ is a net in $Y$ converge to $y^{**}\in Y^{**}$ in the $w^{*}-$topologies, then we have

\begin{eqnarray*}
&&\langle f^{****s**t}(y^{**},x^{**},w^{*}),z^{**}\rangle=\langle f^{****s**}(x^{**},w^{*},y^{**}),z^{**}\rangle\\
&&=\langle f^{****s*}(w^{*},y^{**},z^{**}),x^{**}\rangle=\langle w^{*},f^{****s}(y^{**},z^{**},x^{**})\rangle\\
&&=\langle f^{****}(x^{**},y^{**},z^{**}),w^{*}\rangle=\langle f^{s****t}(x^{**},y^{**},z^{**}),w^{*}\rangle\\
&&=\langle f^{s****}(y^{**},z^{**},x^{**}),w^{*}\rangle=\langle y^{**},f^{s***}(z^{**},x^{**},w^{*})\rangle\\
&&=\lim\limits_\beta \langle f^{s***}(z^{**},x^{**},w^{*}),y_{\beta}\rangle=\lim\limits_\beta \langle z^{**},f^{s**}(x^{**},w^{*},y_{\beta})\rangle\\
&&=\lim\limits_\beta \langle z^{**},f^{s**t}(y_{\beta},x^{**},w^{*})\rangle=\lim\limits_\beta \langle f^{s**t*}(z^{**},y_{\beta},x^{**}),w^{*}\rangle\\
&&=\lim\limits_\beta \langle f^{s**t**}(w^{*},z^{**},y_{\beta}),x^{**}\rangle=\lim\limits_\beta \langle f^{s**t***}(x^{**},w^{*},z^{**}),y_{\beta}\rangle\\
&&=\langle y^{**},f^{s**t***}(x^{**},w^{*},z^{**})\rangle=\langle f^{s**t****}(y^{**},x^{**},w^{*}),z^{**}\rangle.
\end{eqnarray*}
Therefore $f^{****s**t}(Y^{**},X^{**},W^{*})=f^{s**t****}(Y^{**},X^{**},W^{*})$. A similar argument shows that $f^{t**s****}(W^{*},Z^{**},Y^{**})=f^{****t**s}(W^{*},Z^{**},Y^{**})$.\\
Conversely, suppose that 
$$f^{****s**t}(Y^{**},X^{**},W^{*})=f^{s**t****}(Y^{**},X^{**},W^{*})\Rightarrow f^{****}=f^{s****t},$$
$$f^{t**s****}(W^{*},Z^{**},Y^{**})=f^{****t**s}(W^{*},Z^{**},Y^{**})\Rightarrow f^{****}=f^{t****s}.$$
Therefore $f^{****}=f^{s****t}=f^{t****s}$. So by theorem \ref{2.1} completes the proof.
 \end{proof}

\begin{theorem}\label{3.6}
Let $X, Y, Z, W$ and $S$ be normed spaces,  $f:X\times Y\times Z\longrightarrow W$ be a bounded tri-linear mapping and $m:X\times Y\longrightarrow S$ and $g:S\times Z\longrightarrow W$ be bounded bilinear mappings, defined by $f(x,y,z)=g(m(x,y),z)$  for each $x\in X, y\in Y$ and $z\in Z$. If $m$ and $g$ are Arens regular then  $f$ is Aron-Berner regular.
\end{theorem}
\begin{proof}
Assume $(x_{\alpha}), (y_{\beta})$ and $(z_{\gamma})$ are nets in $X, Y$ and $Z$ converge to $x^{**}\in X^{**},y^{**}\in Y^{**}$ and $z^{**}\in Z^{**}$  in the $w^{*}-$topologies, respectively. Then for everey $w^{*}\in W^{*}$ we have
\begin{eqnarray*}
\langle f^{i****i}(x^{**},y^{**},z^{**}),w^{*}\rangle &=&\lim\limits_\beta\lim\limits_\alpha\lim\limits_\gamma \langle f(x_{\alpha},y_{\beta},z_{\gamma}),w^{*}\rangle\\
&=&\lim\limits_\beta\lim\limits_\alpha\lim\limits_\gamma \langle g(m(x_{\alpha},y_{\beta}),z_{\gamma}),w^{*}\rangle\\
&=&\lim\limits_\beta\lim\limits_\alpha\lim\limits_\gamma \langle g^{*}(w^{*},m(x_{\alpha},y_{\beta})),z_{\gamma}\rangle\\
&=&\lim\limits_\beta\lim\limits_\alpha\langle z^{**},g^{*}(w^{*},m(x_{\alpha},y_{\beta})) \rangle\\
&=&\lim\limits_\beta\lim\limits_\alpha\langle g^{**}(z^{**},w^{*}),m(x_{\alpha},y_{\beta}) \rangle\\
&=&\lim\limits_\beta\lim\limits_\alpha\langle g^{**}(z^{**},w^{*}),m^{r}(y_{\beta},x_{\alpha}) \rangle\\
&=&\lim\limits_\beta\lim\limits_\alpha\langle m^{r*}(g^{**}(z^{**},w^{*}),y_{\beta}),x_{\alpha}\rangle\\
&=&\lim\limits_\beta\langle x^{**},m^{r*}(g^{**}(z^{**},w^{*}),y_{\beta})\rangle\\
&=&\lim\limits_\beta\langle m^{r**}(x^{**},g^{**}(z^{**},w^{*})),y_{\beta}\rangle\\
&=&\langle y^{**},m^{r**}(x^{**},g^{**}(z^{**},w^{*}))\rangle\\
&=&\langle m^{r***}(y^{**},x^{**}),g^{**}(z^{**},w^{*})\rangle\\
&=&\langle m^{r***r}(x^{**},y^{**}),g^{**}(z^{**},w^{*})\rangle\\
&=&\langle g^{***}(m^{r***r}(x^{**},y^{**}),z^{**}),w^{*}\rangle.
\end{eqnarray*}
Therefore $$f^{i****i}(X^{**},Y^{**},Z^{**})=g^{***}(m^{r***r}(X^{**},Y^{**}),Z^{**})\ \ \ \ \ \  (2-4)$$
On the other hand,
\begin{eqnarray*}
\langle f^{r****r}(x^{**},y^{**},z^{**}),w^{*}\rangle &=&\lim\limits_\gamma\lim\limits_\beta\lim\limits_\alpha \langle f(x_{\alpha},y_{\beta},z_{\gamma}),w^{*}\rangle\\
&=&\lim\limits_\gamma\lim\limits_\beta\lim\limits_\alpha \langle g(m(x_{\alpha},y_{\beta}),z_{\gamma}),w^{*}\rangle\\
&=&\lim\limits_\gamma\lim\limits_\beta\lim\limits_\alpha \langle w^{*},g^{r}(z_{\gamma},m(x_{\alpha},y_{\beta}))\rangle\\
&=&\lim\limits_\gamma\lim\limits_\beta\lim\limits_\alpha \langle g^{r*}(w^{*},z_{\gamma}),m(x_{\alpha},y_{\beta})\rangle\\
&=&\lim\limits_\gamma\lim\limits_\beta\lim\limits_\alpha \langle g^{r*}(w^{*},z_{\gamma}),m^{r}(y_{\beta},x_{\alpha})\rangle
\end{eqnarray*}
\begin{eqnarray*}
&=&\lim\limits_\gamma\lim\limits_\beta\lim\limits_\alpha \langle m^{r*}(g^{r*}(w^{*},z_{\gamma}),y_{\beta}),x_{\alpha}\rangle\\
&=&\lim\limits_\gamma\lim\limits_\beta\langle x^{**},m^{r*}(g^{r*}(w^{*},z_{\gamma}),y_{\beta})\rangle\\
&=&\lim\limits_\gamma\lim\limits_\beta\langle m^{r**}(x^{**},g^{r*}(w^{*},z_{\gamma})),y_{\beta}\rangle\\
&=&\lim\limits_\gamma\langle y^{**},m^{r**}(x^{**},g^{r*}(w^{*},z_{\gamma}))\rangle\\
&=&\lim\limits_\gamma \langle m^{r***}(y^{**},x^{**}),g^{r*}(w^{*},z_{\gamma})\rangle\\
&=&\lim\limits_\gamma \langle m^{r***r}(x^{**},y^{**}),g^{r*}(w^{*},z_{\gamma})\rangle\\
&=&\lim\limits_\gamma \langle g^{r**}(m^{r***r}(x^{**},y^{**}),w^{*}),z_{\gamma}\rangle\\
&=&\langle z^{**},g^{r**}(m^{r***r}(x^{**},y^{**}),w^{*})\rangle\\
&=&\langle g^{r***}(z^{**},m^{r***r}(x^{**},y^{**})),w^{*}\rangle\\
&=&\langle g^{r***r}(m^{r***r}(x^{**},y^{**}),z^{**}),w^{*}\rangle.
\end{eqnarray*}
Hence
$$f^{r****r}(X^{**},Y^{**},Z^{**})=g^{r***r}(m^{r***r}(X^{**},Y^{**}),Z^{**})\ \ \ \ \ \  (2-5)$$
Finally,

\begin{eqnarray*}
\langle f^{j****j}(x^{**},y^{**},z^{**}),w^{*}\rangle &=&\lim\limits_\alpha\lim\limits_\gamma\lim\limits_\beta \langle f(x_{\alpha},y_{\beta},z_{\gamma}),w^{*}\rangle\\
&=&\lim\limits_\alpha\lim\limits_\gamma\lim\limits_\beta \langle g(m(x_{\alpha},y_{\beta}),z_{\gamma}),w^{*}\rangle\\
&=&\lim\limits_\alpha\lim\limits_\gamma\lim\limits_\beta \langle w^{*},g^{r}(z_{\gamma},m(x_{\alpha},y_{\beta}))\rangle\\
&=&\lim\limits_\alpha\lim\limits_\gamma\lim\limits_\beta \langle g^{r*}(w^{*},z_{\gamma}),m(x_{\alpha},y_{\beta})\rangle\\
&=&\lim\limits_\alpha\lim\limits_\gamma\lim\limits_\beta \langle m^{*}(g^{r*}(w^{*},z_{\gamma}),x_{\alpha}),y_{\beta}\rangle\\
&=&\lim\limits_\alpha\lim\limits_\gamma\langle y^{**},m^{*}(g^{r*}(w^{*},z_{\gamma}),x_{\alpha})\rangle\\
&=&\lim\limits_\alpha\lim\limits_\gamma\langle m^{**}(y^{**},g^{r*}(w^{*},z_{\gamma})),x_{\alpha}\rangle\\
&=&\lim\limits_\alpha\lim\limits_\gamma\langle m^{***}(x_{\alpha},y^{**}),g^{r*}(w^{*},z_{\gamma})\rangle\\
&=&\lim\limits_\alpha\lim\limits_\gamma\langle g^{r**}(m^{***}(x_{\alpha},y^{**}),w^{*}),z_{\gamma}\rangle\\
&=&\lim\limits_\alpha\langle z^{**},g^{r**}(m^{***}(x_{\alpha},y^{**}),w^{*})\rangle
\end{eqnarray*}
\begin{eqnarray*}
&=&\lim\limits_\alpha\langle g^{r***}(z^{**},m^{***}(x_{\alpha},y^{**})),w^{*}\rangle\\
&=&\lim\limits_\alpha\langle g^{r***r}(m^{***}(x_{\alpha},y^{**}),z^{**}),w^{*}\rangle\\
&=&\lim\limits_\alpha\langle g^{***}(m^{***}(x_{\alpha},y^{**}),z^{**}),w^{*}\rangle\\
&=&\lim\limits_\alpha\langle m^{***}(x_{\alpha},y^{**}),g^{**}(z^{**},w^{*})\rangle\\
&=&\lim\limits_\alpha\langle x_{\alpha}, m^{**}(y^{**},g^{**}(z^{**},w^{*}))\rangle\\
&=&\langle x^{**}, m^{**}(y^{**},g^{**}(z^{**},w^{*}))\rangle\\
&=&\langle  m^{***}(x^{**},y^{**}),g^{**}(z^{**},w^{*})\rangle\\
&=&\langle  g^{***}(m^{***}(x^{**},y^{**}),z^{**}),w^{*}\rangle.
\end{eqnarray*}
Therefore
$$f^{j****j}(X^{**},Y^{**},Z^{**})=g^{***}(m^{***}(X^{**},Y^{**}),Z^{**})\ \ \ \ \ \  (2-6)$$
The maps $m$ and $g$ are Arens regular, so by comparing equations (2-4), (2-5) and (2-6),  we conclude  that $f^{i****i}=f^{j****j}=f^{r****r}$. Now by theorem \ref{2.1} proof follows.
\end{proof}



\begin{thebibliography}{99}
\setlength{\baselineskip}{.45cm}




\bibitem{Arens2} {\it R . Arens}, The adjoint of a bilinear operation, Proc. Amer. Math. Soc, {\bf 2}(1951), 839-848.


\bibitem{Arikan1} {\it N. Arikan}, Arens regularity and reflexivity,  Quart. J. Math. Oxford, {\bf 32}(1981), no. 4, 383-388.

\bibitem{Civin and Yood} {\it P. Civin and B. Yood}, The second conjugate space of a Banach algebra as an algebra, Pacific. J. Math, {\bf 11}(1961), no. 3, 847-870.


\bibitem{Eshaghi} {\it M. Eshaghi Gordji and M. Filali}, Arens regularity of module actions, Studia Math, {\bf 181}(2007), no. 3, 237-254.


\bibitem{Haghnejad Azar} {\it K. Haghnejad Azar}, Arens regularity of bilinear forms and unital Banach module space, Bull. Iranian Math. Soc, {\bf 40}(2014), no. 2, 505-520.

\bibitem{Khosravi} {\it A. A. Khosravi, H.R. E. Vishki and A.M. Peralta}, Aron-Berner extensions of triple maps with application to the bidual of Jordan Banach triple systems, Linear Algebra Appl, {\bf 580} (2019), 436-463.

\bibitem{Mohamadzadeh} {\it S. Mohamadzadeh and H. R. E Vishki}, Arens regularity of module actions and the second adjoint of a drivation,  Bull Austral. Mat. Soc, {\bf 77} (2008), 465-476.

\bibitem{ulger}
{\it A. \H{U}lger}, Weakly compact bilinear forms and Arens regularity, Proc. Amer, Math. Soc, {\bf 101}(1987), no. 4, 697-704.  

\bibitem{sheikh}
{\it A. Sheikhali, A. Sheikhali and N. Akhlaghi}, Arens regularity of Banach module actions and the strongly irregular property,  J. Math. Computer Sci, {\bf 13}(2014), no. 1, 41-46.
 
\bibitem{sheikhoo}
{\it A. Sheikhali, K. Haghnejad Azar and A. Ebadian}, Close-to-regularity of bounded tri-linear maps , Global Analysis and Discrete Mathematics,  {\bf 6}(2021), no. 1, 33-39.

\bibitem{sheikhd}
{\it A. Sheikhali, A. Ebadian and  K. Haghnejad Azar}, Regularity of bounded tri-linear maps and the fourth adjoint of a tri-derivation,  Global Analysis and Discrete Mathematics, {\bf 5}(2020), no. 1, 51-65.

\bibitem{young}
{\it N. J. Young}, The irregularity of multiplication in group algebras, Quart. J. Math. Oxford, {\bf 24} (1973), 59-62.

\end{thebibliography}
\end{document}